\documentclass[12pt]{article}
\usepackage{amsmath, amsthm}\usepackage{enumerate}\usepackage{amssymb}\usepackage{bbold}
\usepackage{bbm}\usepackage{dsfont}\usepackage{color}
\newtheorem{theorem}{Theorem}
\newtheorem{proposition}{Proposition}[section]
\newtheorem{problem}{Problem}

\newtheorem{example}{Example}

\newtheorem{remark}{Remark}
\newtheorem{definition}{Definition}[section]

\DeclareMathOperator{\cst}{\xrightarrow[]{\mathbb{st_m}}}

\begin{document}

\title{Statistically multiplicative convergence on locally solid Riesz algebras}\maketitle\author{\centering{{Abdullah Ayd{\i}n$^{1,*}$ and Mikail Et$^{2}$ \\ \small $^1$ Department of Mathematics, Mu\c{s} Alparslan University, Mu\c{s}, 49250, Turkey, \\  \small a.aydin@alparslan.edu.tr \\ \small $^2$ Department of Mathematics, F\i rat University, Elaz\i \u{g}}, 23119, Turkey, \\  \small mikailet68@gmail.com\\$*$Corresponding Author}
		
\abstract{In this paper, we introduce the statistically multiplicative convergent sequences in locally solid Riesz algebras with respect to the algebra multiplication and the solid topology. We study on this concept and we give the notion of $\mathbb{st_m}$-bounded sequence, and also, we prove some relations between this convergence and the other convergences such as the order convergence and the statistical convergence in topological spaces. Also, we give some results related to semiprime $f$-algebras.}\\
\vspace{2mm}

\vspace{2mm}


\section{Introduction and Preliminaries}

The statistical convergence is natural and efficient tools in the theory of functional analysis. It is a generalization of the ordinary convergence of a real sequence; see for example \cite{Fridy-Orhan,M,MK}. The idea of statistical convergence was firstly introduced by Zygmund \cite{Zygmund} in the first edition of his monograph in 1935. Fast \cite{Fast} and Steinhaus \cite{St} independently improved this idea in the same year 1951. Fridy
\cite{Fridy} defined the concept of statistical Cauchiness and showed that it is equivalent to statistical convergence. He also dealt with some Tauberian theorems. Recently, several generalizations and applications of this concept have been investigated by several authors in series of recent papers (c.f. \cite{AAydn1,cinar,ZE,et}).

The concept of statistical boundedness has firstly appeared in the paper of Fridy and Orhan \cite{Fridy-Orhan}. Contrary to statistical convergence it has not drawn that much attention of researchers. That being said Bhardwaj and Gupta \cite{Bhardwaj-Gupta} presented some generalizations of statistical
boundedness by introducing counterparts of the notions of statistical convergence of order $\alpha$ and $\lambda$-statistical convergence. The natural density of subsets of $\mathbb{N}$ plays a critical role in the definition of statistical convergence. For a
subset $A$ of natural numbers if the following limit exists%
\[
\lim_{n\rightarrow \infty}\frac{1}{n}\left \vert \left \{  k\leq n:k\in
A\right \}  \right \vert
\]
then this unique limit is called \emph{the density of }$A$ and mostly abbreviated by $\delta(A)$, where $\left \vert \left \{  k\leq n:k\in A\right \} \right \vert $ is the number of members of $A$ not exceeding $n$. Also, a sequence $x=(x_{k})$ statistically converges to $L$ provided that%
\[
\lim_{n\rightarrow \infty}\frac{1}{n}\left \vert \left \{  k\leq n:\left \vert
x_{k}-L\right \vert \geq \varepsilon \right \}  \right \vert =0 \eqno(1)
\]
for each $\varepsilon>0$. It is written by $S-\lim x_{k}=L$. If $L=0$ then $x$ is a statistically null sequence. A sequence $x=(x_{k})$ is said to be statistically bounded if there exists a
number $B\geq0$ such that $\delta \left(  \{k:|x_{k}|>B\} \right)  =0$. Throughout this paper, the vertical bar of sets will stand for the cardinality of the given sets. 

Riesz space is another concept of functional analysis and it is ordered vector space that has many applications in measure theory, operator theory and applications in economics; see for example \cite{AB,ABPO,Za}. Riesz space was introduced by F. Riesz in \cite{Riez}. A real-valued vector space $E$ with an order relation "$\leq$" on $E$, i.e., it is antisymmetric, reflexive and transitive is called {\em ordered vector space} whenever, for every $x,y\in E$, we have
\begin{enumerate}
	\item[(1)] $x\leq y$ implies $x+z\leq y+z$ for all $z\in E$,
	\item[(2)] $x\leq y$ implies $\lambda x\leq \lambda y$ for every $0\leq \lambda \in \mathbb{R}$.
\end{enumerate}

An ordered vector space $E$ is called {\em Riesz space} or {\em vector lattice} if, for any two vectors $x,y\in E$, the infimum and the supremum
$$
x\wedge y=\inf\{x,y\} \ \ \text{and} \ \ x\vee y=\sup\{x,y\}
$$
exist in $E$, respectively. A vector lattice is called {\em order complete} if every nonempty bounded above subset has a supremum (or, equivalently, whenever every nonempty bounded below subset has an infimum). For an element $x$ in a vector lattice $E$, {\em the positive part}, {\em the negative part}, and {\em module} of $x$ are respectively 
$$
x^+:=x\vee0, \ \ \ x^-:=(-x)\vee0\ \ and \ \ |x|:=x\vee (-x).
$$ 
In the present paper, the vertical bar $|\cdot|$ of elements of the vector lattices will stand for the module of the given elements. It is clear that the positive and negative parts of a vector are positive. Ercan introduced a characterization of statistical convergence on Riesz spaces in \cite{ZE}. Recall that the order convergence is crucial for the vector lattices. A net $(x_\alpha)_{\alpha\in A}$ in a vector lattice $E$ is  called {\em order convergent} to $x\in E$ whenever the following conditions hold: 
\begin{enumerate}
	\item[(1)] there exists a decreasing net $(y_\beta)_{\beta\in B}\downarrow$ such that $\inf y_\beta=0$ in $E$;
	\item[(2)] for any $\beta$, there is $\alpha_{\beta}$ with $|x_\alpha-x|\le y_\beta$ for all $\alpha\ge \alpha_{\beta}$.
\end{enumerate}

Now we turn our attention to the solid topology on vector lattices. 
By a {\em linear topology} $\tau$ on a vector space $E$, we mean a topology $\tau$ on $E$ that makes the addition and the scalar multiplication are continuous. A {\em topological vector space} $(E,\tau)$ is a vector space $E$ equipped with linear topology $\tau$. Neighborhoods of zero on topological vector spaces will often be referred to as zero neighborhoods. Every zero neighborhood $V$ in topological vector spaces is {\em absorbing}, that is, for every element $x$, there exists a positive real $\lambda>0$ such that $\lambda u\in V$. A linear topology $\tau$ on a vector space $E$ has a base $\mathcal{N}$ for the zero neighborhoods satisfying the following properties;
\begin{enumerate}
	\item[(1)] Each $V\in \mathcal{N}$ is balanced, i.e., $\lambda V\subseteq V$ for all scalar $\lvert \lambda\rvert\leq 1$,
	\item[(2)] For each $V_1,V_2\in \mathcal{N}$, there is $V\in \mathcal{N}$ such that $V\subseteq V_1\cap V_2$,
	\item[(3)] For every $V\in \mathcal{N}$, there exists $U\in \mathcal{N}$ with $U+U\subseteq V$,
	\item[(4)] For any scalar $\lambda$ and each $V\in \mathcal{N}$, the set $\lambda V$ is also in $\mathcal{N}$.
\end{enumerate}

In this article, unless otherwise, when we mention a zero neighborhood, it means that it always belongs to a base that holds the above properties. A subset $A$ of a vector lattice $E$ is called {\em solid} if, for each $x\in A$ and $y\in E$ with  $|y|\leq|x|$ implies $y\in A$. Let $E$ be vector lattice and $\tau$ be a linear topology on it. Then the pair $(E,\tau)$ is said {\em locally solid vector lattice} (or, {\em locally solid Riesz space}) if $\tau$ has a base which consists of solid sets; for much more details on these notions, see \cite{AB,ABPO,Za}. 

The statistical convergence in a locally solid Riesz space was introduced and studied by Ayd\i n  in \cite{AAydn1} with respect to unbounded order convergence on locally solid Riesz spaces. Let $(x_n)$ be a sequence in a locally solid vector lattice $E$ is said to be {\em statistically unbounded $\tau$-convergent} to $x\in E$ if, for every zero neighborhood $U$, it satisfies the following limit 
$$
\lim\limits_{n\to\infty}\frac{1}{n}\big\lvert\{k\leq n:\lvert x_k-x\rvert\wedge u\notin U\}\big\rvert=0 \eqno(2)
$$
for all $u\in E_+$. We abbreviate this convergence as $x_n\xrightarrow{st-u_\tau}x$, or shortly, we say that $(x_n)$ $st$-$u_\tau$-converges to $x$.

In this paper, we aim to introduce the concept of statistical convergence on locally solid Riesz algebras. So, let give some basic notations and terminologies of Riesz algebras that will be used. A Riesz space $E$ is called a {\em Riesz algebra} (or {\em $l$-algebra}, for short), if $E$ is an associative algebra whose positive cone $E_+=\{x\in E:x\geq0\}$ is closed under the algebra multiplication, i.e., 
$$
x,y\in E_+ \ \ \Rightarrow \ \ x\cdot y\in E_+.
$$

We routinely use the following fact: $y\leq x$ implies $u\cdot y\leq u\cdot x$ for all arbitrary elements $x,y$ and for every positive element $u$ in $l$-algebras. We refer the reader for an exposition on the following notions to \cite{ABPO,AEG,HP,Pag,Za}.
\begin{definition}\label{various lattice algebras}
	An $l$-algebra $E$ is said to
	\begin{enumerate}
		\item[(1)] \ {\em $d$-algebra} whenever $u\cdot(x\wedge y)=(u\cdot x)\wedge(u\cdot y) \ \ \text{and} \ \ (x\wedge y)\cdot u=(x\cdot u)\wedge(y\cdot u)$ holds for all $u,x,y\in E_+$;
		
		\item[(2)] \ {\em almost $f$-algebra} if $x\wedge y=0$ implies $x\cdot y=0$ for all $x,y\in E_+$;
		
		\item[(3)] \ {\em $f$-algebra} if $x\wedge y=0$ implies $(u\cdot x)\wedge y=(x\cdot u)\wedge y=0$ for all $u,x,y\in E_+$;
		
		\item[(4)] \ {\em semiprime} whenever the only nilpotent element in $E$ is zero, where an element $x$ in $E$ is called {\em nilpotent} whenever $x^n=0$ for some natural number $n\in \mathbb{N}$;
		
		\item[(5)] \ {\em unital} if $E$ has a multiplicative unit;
		
		\item[(6)] \ have {\em the factorization property} if, for given $x\in E$, there exist $y,z\in E$ such that $x=y\cdot z$.
	\end{enumerate}
\end{definition}

The important observation on the Definition \ref*{various lattice algebras} is that every $f$-algebra is both $d$- and almost $f$-algebra. A vector lattice $E$ is called \textit{Archimedean} whenever $\frac{1}{n}x\downarrow 0$ holds in $E$ for each $x\in E_+$. Every Archimedean $f$-algebra is commutative; see \cite[Thm.140.10]{Za}. 

Assume $E$ is an Archimedean $f$-algebra with a multiplicative unit vector $e$. Then, by applying \cite[Thm.142.1.(v)]{Za}, it can be seen that $e$ is a positive vector. In arbitrary $l$-algebras the unit element (if existing) need not be positive; see for example \cite[Exam.1.2(iii)]{HP}. In this article, unless otherwise, all vector lattices are assumed to be real and Archimedean. Different types of convergences in Riesz algebras have attracted the attention of several authors in a series of recent papers; see for example \cite{AAydn2,AAydn3,AEG}.

\begin{example}
	Let $E$ be a vector lattice. Consider $Orth(E):=\{\pi\in L_b(E):x\perp y\ \text{implies}\ \pi x\perp y\}$ the set of orthomorphisms on $E$, i.e., $|x|\wedge|y|=0$ in $E$ implies $|\pi x|\wedge|y|=0$. Then, for an order complete $E$, $\text{Orth}(E)$ is a order complete unital $f$-algebra; see $\cite[Thm.15.4]{Pag}$.
\end{example}

\section{Definitions and basic properties of the $\mathbb{st_m}$-convergence}
In this section, we introduce the $\mathbb{st_m}$-convergence and we give some basic results. Let's start with the asymptotic density that is great importance for the statistical convergences. Let $K$ be a subset of the set $\mathbb{N}$ of all natural numbers. 
Define a new set $K_n=\{k\in K:k\leq n\}$. If the limit of 
$$
\mu(K):=\lim\limits_{n\to\infty}\lvert K_n\rvert/n
$$
exists then $\mu(K)$ is called the {\em asymptotic density} of the set $K$. Recall that $\mu(\mathbb{N}-A)=1-\mu(A)$ for $A\subseteq\mathbb{N}$. Let $X$ be a topological space and $(x_n)$ be a sequence in $X$. Recall that $(x_n)$ is said to be {\em statistically convergent in topology} to $x\in X$ whenever, for each neighborhood $U$ of $x$, we have $\mu\big(\{n\in\mathbb{N}:x_n\notin U\}\big)=0$; see for example \cite{M,MK}. 

\begin{definition}
	Let $E$ be $l$-algebra and $(E,\tau)$ be a locally solid vector lattice. Then the pair $(E,\tau)$ is called {\em locally solid $l$-algebra}.
\end{definition}

We continue with the following notion which is crucial for the present paper.

\begin{definition}
	Let $(E,\tau)$ be a locally solid $l$-algebra and $(x_n)$ be a sequence in $E$. Then $(x_n)$ is said to be {\em statistically multiplicative convergent} to $x\in E$ if, for every zero neighborhood $U$, the following limit 
	$$
	\lim\limits_{n\to\infty}\frac{1}{n}\big\lvert\{k\leq n:u\cdot|x_k-x|\notin U\}\big\rvert=0 \eqno(3)
	$$
	holds for all $u\in E_+$. 
\end{definition}

We abbreviate this convergence as $x_n\cst x$, or we say that $(x_n)$ {\em $\mathbb{st_m}$-converges} to $x$, for short. Let's take $K=\{n\in \mathbb{N}:u\cdot \lvert x_k-x\rvert \notin U\}$ for arbitrary zero neighborhood $U$ and for each positive vector $u$ in a locally solid $l$-algebra $E$. Then, briefly, $(x_n)$ is $\mathbb{st_m}$-convergent to $x\in E$ whenever $\mu(K)=0$ for every $u\in E_+$. Moreover, the limit in $(3)$ can be also defined for the right multiplicative or both sides. But, to simplify the presentation, we suppose that Riesz algebras are under consideration to be commutative. The first observation is in the following remark.

\begin{remark}
	Let $(E,\tau)$ be a positive unital locally solid $l$-algebra. Then the statistically multiplicative convergence implies the statistical convergence in topology. Indeed, let $e\geq0$ be the positive multiplicative unit element of $E$. Then it follows from the positivity of $e$ and the equality $e\cdot|x_n-x|=|x_n-x|$ that one can get the desired result.
\end{remark}

We continue with several basic and useful results that are motivated by their analogies from vector lattice theory, and also, they are similar to the classical results for so many kinds of statistical convergences such as \cite[Thm.2.2.]{AAydn1} and similar to \cite[Thm.2.17.]{AB}, so their proofs are omitted.

\begin{theorem}\label{basic properties}
	Let $x_n\cst x$ and $y_n\cst y$ be two sequences in a locally solid $l$-algebra $(E,\tau)$. Then, we have
	\begin{enumerate}
		\item[(i)] \ $x_n\vee y_n\cst x\vee y$;
		\item[(ii)] \ $x_n\cst x$ iff  $(x_n-x)\cst0$ iff $\lvert x_n-x\rvert\cst0$;
		\item[(iii)] \ $rx_n+sy_n\cst rx+sy$ for any $r,s\in{\mathbb R}$;
		\item[(iv)]\ $x_{n_k}\cst x$ for any subsequence $(x_{n_k})$ of $x_n$.
	\end{enumerate}
	In particular, we have
	$$
	x_n^+\cst x^+,\ \ x_n^-\cst x^- \ \ and \ \ |x_n|\cst |x|.
	$$
\end{theorem}

\begin{proposition}\label{clasiqal unique limit}
	Let $(E,\tau)$ be a Hausdorff locally solid $l$-algebra. Then limit of the $\mathbb{st_m}$-convergence is uniquely determined in $E$.
\end{proposition}

The following result show that the $\mathbb{st_m}$-convergence satisfies normality in topological vector spaces.
\begin{proposition}\label{normality of st conv}
	Let $(x_n)$ and $(y_n)$ be two sequences in a locally solid $l$-algebra $(E,\tau)$. If $x_n\cst 0$ and $0\leq y_n\leq x_n$ for all $n\in\mathbb{N}$ then $y_n\cst 0$. 
\end{proposition}

\begin{proof}
	Let $U$ be an arbitrary zero neighborhood in $E$. Then we have $\mu(K)=0$, where $K=\{n\in\mathbb{N}:u\cdot x_k\notin U\}$ for each $u\in E_+$ because $(x_n)$ is positive and $x_n\cst 0$. 
	By using the solidness of $U$, we get $u\cdot y_k\in U$ for each $u\in E_+$ because of $u\cdot y_k\leq u\cdot x_k\in U$ for all $u\in E_+$. 
	Therefore, we have $\mu\big(\{n\in\mathbb{N}:u\cdot y_k\notin U\}\big)=0$, i.e., we have $y_n\cst 0$ because the set $\{n\in\mathbb{N}:u\cdot y_k\notin U\}$ is a subset of $\{n\in\mathbb{N}:u\cdot x_k\notin U\}$ and $U$ is arbitrary zero neighborhood.
\end{proof}

\begin{definition}
	Let $(E,\tau)$ be a locally solid $l$-algebra. A sequence $(x_n)$ in $E$ is said to be {\em $\mathbb{st_m}$-Cauchy} if $(x_n-x_m)_{(m,n)\in\mathbb{N}\times\mathbb{N}}\cst 0$. Also, $E$ is said to be {\em $\mathbb{st_m}$-complete} if every $\mathbb{st_m}$-Cauchy sequence in $E$ is $\mathbb{st_m}$-convergent. 
\end{definition}

\begin{proposition}
	Every $\mathbb{st_m}$-convergent sequence in a locally solid $l$-algebra is $\mathbb{st_m}$-Cauchy.
\end{proposition}

\begin{proof}
	Suppose that a sequence $(x_n)$ is $\mathbb{st_m}$-convergent to a point $x\in E$. Take an arbitrary zero neighborhood $U$. Then there is another zero neighborhood $V$ such that $V+V\subseteq U$. Thus we have $\mu(K)=1$, where $K=\{n\in\mathbb{N}:u\cdot|x_n-x|\in V\}$ for every $u\in E_+$. Thus, for each $u\in E_+$, we can obtain the inequality
	$$
	u\cdot|x_n-x_m|\leq u\cdot|x_n-x|+u\cdot|x_m-x|\in V+V\subseteq U.
	$$
	for all $n,m\in K$. Hence we get $K\subseteq \{n\in \mathbb{N}:u\cdot|x_n-x_m|\in U\}$, and so, $\mu\big(\{n\in \mathbb{N}:u\cdot|x_n-x_m|\in U\}\big)=1$ holds for every $u\in E_+$. Therefore, we get the desired, being $\mathbb{st_m}$-Cauchy of $(x_n)$, result.
\end{proof}

\begin{definition}\label{bounded sequence}
	A sequence $(x_n)$ in a locally solid $l$-algebra is called {\em statistically multiplicative bounded} (or {\em $\mathbb{st_m}$-bounded, for short}) whenever, for each zero neighborhood $U$, there exists a positive element $w\in E_+$ such that
	$$
	\mu\big(\{n\in\mathbb{N}:w.|x_n|\notin U\}\big)=0.
	$$
\end{definition}

It is clear that statistically multiplicative bounded sequence coincides with the notion the statistically bounded for real sequences (c.f. \cite{Fridy-Orhan}).

\begin{theorem}\label{convergent sequence bounded}
	An $\mathbb{st_m}$-convergent sequence $(x_n)$ in a locally solid $l$-algebra is statistically multiplicative bounded.
\end{theorem}

\begin{proof}
	Suppose $(E,\tau)$ is a locally solid $f$-algebra. Then since the topology $\tau$ is a locally solid topology on $E$, it follows from \cite[Thm.4.]{AEG} and \cite[Lem.2.1]{AEG} that the topology $\tau$ is locally full in $E$ which is the condition of \cite[Thm.17.]{AEG}. Now, we can apply it. Let $\mathcal{N}_0$ be the base of zero $\tau$-neighborhoods. Thus, we can define a new zero neighborhood
	$$
	W_{u,U}:=\{x\in E: u\cdot|x|\in E\} \eqno(4)
	$$
	for any element $u\in E_+$ and any zero neighborhood $U\in \mathcal{N}_0$. Then the class $\mathcal{B}:=\{W_{u,U}: u\in E_+, U\in\mathcal{T}\}$ is a base of zero neighborhood of locally solid topology $\tau_m$ on $E$; see \cite[Thm.17.]{AEG}.
	
	Now, assume that a sequence $x_n\cst x$ in $E$. Let $U$ be an arbitrary zero $\tau$-neighborhood in $\mathcal{N}_0$. 
	Then there exists another $V\in \mathcal{N}_0$ such that $V+V\subseteq U$. 
	By using the $\mathbb{st_m}$-convergence of $(x_n)$, we have $\mu(K)=0$, where $K=\{n\in\mathbb{N}:u\cdot|x_n-x|\notin V\}$ for all $u\in E_+$. 
	Let's take a fixed positive element $u\in E_+$. Then we have $\mu(K_W)=0$ for the set $K_W=\{n\in\mathbb{N}:(x_n-x)\notin W_{u,V}\}$. 
	On the other hand, since every linear topology is absorbing, there exists a positive scalar $\lambda>0$ such that $\lambda x\in W_{u,V}$. Now, take another scalar $\gamma>0$ such that $\gamma\leq 1$ and $\gamma\leq\lambda$. Hence, we have $\gamma x\in W_{u,V}$ because $W_{u,V}$ is solid set and $\gamma|x|\leq\lambda|x|\in W_{u,V}$. 
	
	Next, by using $V+V\subseteq U$, one can see that $W_{u,V}+W_{u,V}\subseteq W_{u,U}$ for all $u\in E_+$. So, we have
	$$
	\gamma x_n=\gamma(x_n-x)+\gamma x\in W_{u,V}+W_{u,V}\subseteq W_{u,U}.
	$$
	for every $n\in\mathbb{N}-K_W$, and so, we get $\mu\big(\{n\in\mathbb{N}:\gamma x_n\notin W_{u,U}\}\big)=0$. If we choose the constant element $w$ in Definition \ref{bounded sequence} as $w:=\gamma u$ for the given $u$ then we have $w\in E_+$ and $\mu\big(\{n\in\mathbb{N}:w\cdot|x_n|\notin U\}\big)=0$. Therefore, we get the desired, the $\mathbb{st_m}$-boundedness of $(x_n)$, result because $U$ is arbitrary zero neighborhood.
\end{proof}

\section{Main results in Riesz algebras}
In this section, we give some results of the $\mathbb{st_m}$-convergence that are related to the multiplication on Riesz algebras. We begin with the following useful result for which we will use the basic properties in Theorem \ref{basic properties} and \cite[Prop.2.2.]{AEG}.
\begin{theorem}\label{inequalities}
	Let $(E,\tau)$ be a locally solid $l$-algebra. Then the following facts are equivalent$:$
	\begin{enumerate}
		\item[$(i)$] \  the $\mathbb{st_m}$-convergence has a unique limit;
		\item[$(ii)$] \ if $y_n\leq x_n$ for all $n\in \mathbb{N}$, $x_n\cst x$ and $y_n\cst y$ then $y\leq x$;
		\item[$(iii)$] \ the positive cone $E_+$ is closed under the $\mathbb{st_m}$-convergence in $E$.
	\end{enumerate}
\end{theorem}

\begin{proof}
	$(i)\Longleftrightarrow(ii)$ \ Suppose that $x_n\cst x$, $y_n\cst y$, and $y_n\leq x_n$ for all $n\in \mathbb{N}$. Then we get $(x_n-y_n)\cst (x-y)$, and also, we have
	$$
	x_n-y_n=(x_n-y_n)^+\cst (x-y)^+\geq 0
	$$
	because of $y_n\leq x_n$ for each $n\in \mathbb{N}$. Since the $\mathbb{st_m}$-convergence has a unique limit, we have $x-y=(x-y)^+\geq 0$, i.e., we get $x\geq y$.
	
	For the converse fact, we assume that the $\mathbb{st_m}$-convergence does not have a unique limit, on the contrary case. Suppose that there exists a sequence $(x_n)$ in $E$ and $x,y\in E$ such that $x_n\cst x$, $x\neq y$ and $x_n\cst y$. Thus, we have $0=x_n-x_n\cst x-y$. As a result, we have proved that the zero sequence is $\mathbb{st_m}$-convergent to $|x-y|$, i.e., we have $0\cst|x-y|$. Now, we consider two sequences $y_n=z_n=0$ in $E$. Then we can say $y_n\cst y:=0$ and $z_n\cst z:=|x-y|\neq0$. Thus, one can observe $y\ngeq z$, which disrupts the condition $(ii)$. Therefore, we get the desired a contradiction that proves the uniqueness limit of the $\mathbb{st_m}$-convergence.
	
	$(ii)\Longleftrightarrow(iii)$ \ To show the closedness of the positive cone $E_+$ under the $\mathbb{st_m}$-convergence, one can consider the following fact $0=:y_n\leq x_n\cst x$ implies $x\geq0$.
	
	Conversely, suppose $x_n\cst x$, $y_n\cst y$, and $y_n\leq x_n$ for all $n\in \mathbb{N}$. Then we have $(x_n-y_n)\cst (x-y)$. Since $E_+$ is closed under the $\mathbb{st_m}$-convergence in $E$ and $x_n-y_n\in E_+$ for each $n\in \mathbb{N}$, we get $x-y\in E_+$, i.e., we have $x\ge y$.
\end{proof}

The following work gives a relation between the $\mathbb{st_m}$-convergence and the order convergence.

\begin{proposition}
	Every monotone $\mathbb{st_m}$-convergent sequence is order convergent to its $\mathbb{st_m}$-limit point in locally solid $l$-algebras whenever the $\mathbb{st_m}$-convergence has a unique limit.
\end{proposition}

\begin{proof}
	Modify Proposition 2.6. in \cite{AAydn1} and applying Theorem \ref{inequalities}.
\end{proof}

In the next result, we prove that the algebra multiplication is statistically multiplicative continuous in the following sense.
\begin{theorem}\label{multiplicative with an element again gonvergent}
	Let $(x_n)$ be an $\mathbb{st_m}$-convergent sequence to a point $x$ in a locally solid $l$-algebra $E$. Then $y\cdot x_n\cst y\cdot x$ for all $y\in E$. 
\end{theorem}

\begin{proof}
	Suppose $x_n\cst x$ in $E$ and $y\in E$. So, by applying Theorem \ref{basic properties}, we can get $(x_n-x)^{+}\cst 0$ and $(x_n-x)^{-}\cst 0$. Then one can get $y^+\cdot(x_n-x)^{\pm}\cst 0$ and $y^-\cdot(x_n-x)^{\pm}\cst 0$. Indeed, consider a zero neighborhood $U$. If we choose $w_u:=u\cdot y^+\in E_+$ for every $u\in E_+$ then we have 
	$$
	u\cdot|y^+\cdot(x_n-x)^{\pm}|=u\cdot y^+\cdot(x_n-x)^{\pm}=w_u\cdot(x_n-x)^{\pm}
	$$
	because the negative and the positive part of elements in vector lattices are positive. Therefore, we have $\mu\big(\{n\in\mathbb{N}:u\cdot|y^+\cdot(x_n-x)^{\pm}|\notin U\}\big)=0$ for all $u\in E_+$ because of $\mu\big(\{n\in\mathbb{N}:w_u\cdot(x_n-x)^{\pm}\notin U\}\big)=0$ for all $w_u\in E_+$, i.e., we have $y^+\cdot(x_n-x)^{\pm}\cst 0$, where we use ${\pm}$ for both cases the negative and the positive parts of $(x_n-x)$. Similarly, one can show $y^-\cdot(x_n-x)^{\pm}\cst 0$.
	
	Now, by applying the equality $x=x^+-x^-$ for any element in vector lattices, we can consider the following equality
	\begin{eqnarray*}
		y\cdot(x_n-x)&=&(y^+-y^-)\cdot\big[(x_n-x)^+-(x_n-x)^-\big]\\&=&y^+\cdot(x_n-x)^+-y^+\cdot(x_n-x)^-\\&&-\ y^-\cdot(x_n-x)^++y^-\cdot(x_n-x)^-.
	\end{eqnarray*}
	Thus, we can say $y\cdot(x_n-x)\cst 0$ because of $y^+\cdot(x_n-x)^+$, $y^+\cdot(x_n-x)^-$, $y^-\cdot(x_n-x)^+$ and $y^-\cdot(x_n-x)^-$ are $\mathbb{st_m}$-convergent to zero. Thus, we get the desired result, i.e., we have $y\cdot x_n\cst y\cdot x$.
\end{proof}

It is known that the product of $mo$-convergent nets in $f$-algebras is $mo$-convergent (see \cite{AAydn2}) and the product of $\mathbb{mc}$-convergent nets in commutative $l$-algebras is $\mathbb{mc}$-convergent (see \cite{AEG}). Fortunately, we also have a positive answer for the $\mathbb{st_m}$-convergence in Theorem \ref{product of st convergence}. But, the product of statistically multiplicative bounded sequences does not need to be $\mathbb{st_m}$-bounded in general because the product of zero neighborhood sets can be bigger than them. 

\begin{problem}
	Let $(x_n)$ and $(y_n)$ are $\mathbb{st_m}$-bounded sequences in a locally solid $l$-algebra with the factorization property. Is the product of $(x_n\cdot y_n)$ $\mathbb{st_m}$-bounded sequence?
\end{problem}

Now, we give the following notion to give a partial answer of the product bounded sequences.

\begin{definition}
	Let $E$ be a locally solid $l$-algebra. Then $E$ is said to have {\em the neighborhood factorization property} whenever, for every zero neighborhood $U$ in $E$, there exist neighborhoods $V,W$ with $V\cdot W\subseteq U$.
\end{definition}

It is clear that the real line has the neighborhood factorization property. For the following example, we consider \cite[Exam.11.2.]{Pag}.
\begin{example}
	Let $E$ be the collection of all real-valued continuous functions $f$ which are piecewise polynomials with the norm topology on $[0,1]$. With respect to the pointwise operations, it is an Archimedean $f$-algebra with the unit element $e$ i.e., $e(x)= 1$ for all $0\leq x\leq 1$. Then $E$ has the factorization property, and also, it has the neighborhood factorization property.
\end{example}

\begin{proposition}\label{product of bounded sequences}
	If $(x_n)$ and $(y_n)$ are $\mathbb{st_m}$-bounded sequences in a locally solid $l$-algebra with the neighborhood factorization property then the product of $(x_n\cdot y_n)$ is also $\mathbb{st_m}$-bounded sequence.
\end{proposition}

\begin{proof}
	Suppose $(x_n)$ and $(y_n)$ is $\mathbb{st_m}$-bounded in a locally solid $l$-algebra $E$ with the neighborhood factorization property and $U$ is a zero neighborhood in $E$. Then there exist neighborhoods $V,W$ such that $V\cdot W\subseteq U$. Since $(x_n)$ and $(y_n)$ is $\mathbb{st_m}$-bounded, there exists positive elements $w_x,w_y\in E_+$ such that 
	$$
	\mu\big(\{n\in\mathbb{N}:w_x\cdot|x_n|\notin V\}\big)=0
	$$
	and
	$$
	\mu\big(\{n\in\mathbb{N}:w_y\cdot|y_n|\notin W\}\big)=0.
	$$
	Now, by taking $w=w_x\cdot w_y$ and applying properties in \cite[p.151]{HP}, we can obtain the following inequality
	$$
	w\cdot|x_n\cdot y_n|\leq (w_x\cdot|x_n|)\cdot (w_y\cdot|y_n|)\in V\cdot W\subseteq U
	$$
	for each $n\in \mathbb{N}$. Therefore, for the given zero neighborhood $U$, we find a positive element $w$ such that $\mu\big(\{n\in\mathbb{N}:w\cdot|x_n\cdot y_n|\notin U\}\big)=0$, i.e., we get the desired, the $\mathbb{st_m}$-boundedness of $(x_n\cdot y_n)$, result. 
\end{proof}

In the next result, we prove that the product of $\mathbb{st_m}$-convergent sequences is convergent.
\begin{theorem}\label{product of st convergence}
	If $x_n\cst x$ and $y_n\cst y$ in a locally solid $l$-algebra $E$ then the product sequence $(x_n\cdot y_n)$ $\mathbb{st_m}$-converges to $x\cdot y$ whenever at least one of these two sequences is order bounded.
\end{theorem}

\begin{proof}
	Suppose $x_n\cst x$ and $y_n\cst y$ in $E$. Without loss of generality, suppose $|x_n|\leq w$ for all $n\in \mathbb{N}$ and some $w\in E_+$. Then, by applying the properties in \cite[P.151]{HP}, we have the following inequality
	\begin{eqnarray*}
		|x_n\cdot y_n-x\cdot y|&=&|x_n\cdot y_n-x_n\cdot y+x_n\cdot y-x\cdot y|\\&\leq& |x_n|\cdot|y_n-y|+|y|\cdot|x_n-x|\\&\leq& v\cdot|y_n-y|+|y|\cdot|x_n-x|
	\end{eqnarray*}
	for all $n\in \mathbb{N}$. Since $x_n\cst x$ and $y_n\cst y$, it follows from Theorem \ref{basic properties} that $|x_n-x|$ and $|y_n-y|$ are $\mathbb{st_m}$-convergent to zero. Thus, by using Theorem \ref{multiplicative with an element again gonvergent}, we get $v\cdot|y_n-y|,|y|\cdot|x_n-x|\cst 0$. Lastly, by using Proposition \ref{normality of st conv} and Theorem \ref{basic properties}, one can get $|x_n\cdot y_n-x\cdot y|\cst 0$, i.e., $x_n\cdot y_n\cst x\cdot y$.
\end{proof}

\begin{problem}
	Is Theorem \ref{product of st convergence} true in the case at least one is $\mathbb{st_m}$-bounded?
\end{problem}

\section{The $\mathbb{st_m}$-convergence on $f$-algebras}
In this section, we prove some results on $f$-algebras. Firstly, we give an extension of Proposition \ref{clasiqal unique limit} in the following work.
\begin{theorem}\label{semiprime f-algebra unique}
	The limit of an $\mathbb{st_m}$-convergent sequence is uniquely determined in locally solid semiprime $f$-algebras.
\end{theorem}

\begin{proof}
	We show firstly that if a sequence $(x_n)$ consists of nilpotents of $E$ such that $x_n\cst x$ then $x$ is also a nilpotent element in $E$. To see this, take a fixed zero neighborhood $U$. Thus, we have $\mu(K)=0$ for the set $K=\{n\in\mathbb{N}:u\cdot \lvert x_k-x\rvert \notin U\}$. Now, we fix a positive element $u\in E_+$. Then following from \cite[Thm.142.1(ii)]{Za} and \cite[Prop.10.2(iii)]{Pag}, we have
	$$
	u\cdot|x_n-x|=|u\cdot x_n-u\cdot x|=|u\cdot x|=u\cdot|x|.\eqno(5)
	$$
	Let's define a new constant sequence $y_n(u):=u\cdot|x|$ for the fixed $u\in E_+$ and for all $n\in\mathbb{N}$. Then it follows from $(5)$ that 
	$$
	\{n\in\mathbb{N}:w\cdot|y_n(u)-0|\notin U\}=\{n\in\mathbb{N}:w\cdot y_n(u)\notin U\}
	$$
	for each $w\in E_+$. Thus, we get $y_n(u)\cst 0$ or $u\cdot|x|\cst0$ because of $\mu(K)=0$. As a result, we see that $u\cdot x=0$ for every $u\in E_+$ because $u\cdot|x|$ is constant and converges to zero for all $u\in E_+$. Then $y\cdot x=0$ for all $y\in E$. Hence, following from \cite[p.157]{HP}, we get that $x$ is a nilpotent element in $E$.
	
	Next, we show that $\mathbb{st_m}$-limit is unique. Assume it is not true, i.e., there exists a sequence $(z_n)$ in $E$ such that $z_n\cst z_1$ and $z_n\to z_2$ in $E$. Consider a zero neighborhood $U$. Then there exists another zero neighborhood $V$ such that $V+V\subseteq U$. Then we have $\mu(K_x)=\mu(K_y)=1$, where $K_x=\{n\in\mathbb{N}:v\cdot \lvert x_k-x\rvert\in V\}$ and $K_y=\{n\in\mathbb{N}:v\cdot \lvert x_k-y\rvert\in V\}$ for all $v\in E_+$. Also, we have
	$$
	|z_1-z_2|\leq|z_n-z_1|+|z_n-z_2|.\eqno(6)
	$$
	By taking $v=|z_1-z_2|\in E_+$ and considering $(6)$, we have 
	$$
	|z_1-z_2|^2\leq|z_1-z_2|\cdot|z_n-z_1|+|z_1-z_2|\cdot|z_n-z_2|\in V+V\subseteq U.
	$$
	Therefore, similar to the first part of the proof, we have $|z_1-z_2|^2=0$. Since $E$ is semiprime $f$-algebra. Then we get $|z_1-z_2|=0$, i.e., we have $z_1=z_2$.
\end{proof}

The following result is the $\mathbb{st_m}$-version of \cite[Thm.11.]{AEG}.
\begin{theorem}
	Let $(E,\tau)$ be a locally solid $f$-algebra. Then the $\mathbb{st_m}$-convergence has a unique limit if and only if $E$ is semiprime.
\end{theorem}

\begin{proof}
	Suppose $E$ is semiprime. Assume the $\mathbb{st_m}$-convergence is not uniquely determined in $E$. Then, following from \cite[Lem.1.1.]{AEG} that there exists a constant sequence $x_n=x$ in $E$ such that $x_n\cst y\neq x$. Thus, for given zero neighborhood $U$, we have $\mu(K)=0$ for the set $K=\{n\in\mathbb{N}:u\cdot|x-y|\notin U\}$ for all $u\in E_+$. By repeating the same argument in the proof of Theorem \ref{semiprime f-algebra unique}, we can get $u\cdot|x-y|=0$ for each $u\in E_+$. In the special, take $u=|x-y|$. Then we have $|x-y|^2=0$, i.e., we get $x=y$ because $E$ is semiprime. But, it contradicts with $x\neq y$, and so, we obtain the desired result by the contradiction.
	
	Conversely, suppose the $\mathbb{st_m}$-convergence has a unique limit point. Assume there exists an element $0\neq x\in E$ such that $x^2=0$. Hence, $x$ is a nilpotent element in $E$; see \cite[Prop.10.2(i)]{Pag}. Then, by applying \cite[Prop.10.2(iii)]{Pag}, we get $u\cdot|x|=0$ for all $u\in E_+$. Therefore, one can get $x\cst 0$. Then it follows from unique $\mathbb{st_m}$-limit and \cite[Lem.1.1.]{AEG} that $x=0$, that is, $E$ is semiprime.
\end{proof}


\end{document}